\documentclass[11pt]{amsart}
\usepackage{amsthm}
\usepackage{titlesec}
\usepackage{mathtools}
\usepackage{enumitem}
\usepackage{graphicx, amsmath,latexsym, amsfonts, amssymb, amstext}
\usepackage{mathrsfs}
\usepackage{color}
\usepackage[utf8]{inputenc}
\usepackage[english]{babel}
\usepackage{hyperref}
\hypersetup{
    colorlinks=true,
    linkcolor=blue,
   filecolor=magenta,      
}

\usepackage{textcomp}
\usepackage[small,bf]{caption}
\usepackage{amssymb, amsmath, amsfonts, amsthm}
\usepackage{epsfig,amsfonts,amsbsy,amssymb,amsthm}
\usepackage[numbers]{natbib}
\allowdisplaybreaks

\newtheorem{theorem}{Theorem}
\newtheorem{lemma}[theorem]{Lemma}
\newtheorem{corollary}[theorem]{Corollary}
\newtheorem{proposition}[theorem]{Proposition}
\newtheorem{oldtheorem}{Theorem}

\theoremstyle{definition}
\newtheorem{question}{Question}

\newtheorem{definition}[theorem]{Definition}
\newtheorem{example}[theorem]{Example}
\newtheorem{remark}[theorem]{Remark}

\numberwithin{theorem}{section}
\usepackage{titlesec}
\titleformat{\section}{\normalfont\bfseries\large}{\S\thesection.}{2mm}{}

\newcommand{\ints}{\mathbb{Z}}

\newcommand{\m}{\mathfrak m}

\def\to{\longrightarrow}
\DeclareMathOperator{\Supp}{Supp}
\DeclareMathOperator{\assh}{Assh}
\DeclareMathOperator{\ass}{Ass}

\DeclareMathOperator{\depth}{depth}

\titlespacing*{\section}{0pt}{2\baselineskip}{1\baselineskip}
\title[On Hilbert coefficients of parameter ideals and Cohen-Macaulayness]{On Hilbert coefficients of parameter ideals and Cohen-Macaulayness}

\author{Kumari Saloni}
\thanks{Date: September 23, 2015.}
\thanks{2000 Mathematics Subject Classification 13H10, 13D40}
\address{Department of Mathematics, IIT Guwahati, Guwahati, Assam 781039, India}
\email{saloni.kumari@iitg.ac.in}

\newcommand*\xbar[1]{%
  \hbox{%
    \vbox{%
      \hrule height 0.5pt 
      \kern0.5ex
   \hbox{%
        \kern-0.1em
        \ensuremath{#1}%
        \kern-0.1em
      }%
    }%
  }%
}

\begin{document}
\sloppy
 \maketitle
%
 \begin{abstract}
Let $(R,\m)$ be an unmixed Noetherian local ring, $Q$ a parameter ideal and $K$ an $\m$-primary ideal of $R$ containing $Q$.  
We give a necessary and sufficient condition for $R$ to be Cohen-Macaulay in terms of 
$g_0(Q)$ and $g_1(Q)$, the Hilbert coefficients of $Q$ with respect to $K$. As a consequence,
we obtain a result of Ghezzi et al. which settles the negativity conjecture of W. V. Vasconcelos \cite{vanishing-conjecture} in unmixed local rings. 
\end{abstract}
 
\section{Introduction}
Let $(R,\m)$ be a Noetherian local ring of dimension $d > 0$ and $Q$ an $\m$-primary ideal of $R$. 
Let $K$ be an ideal such that $Q\subseteq K$. 
Let $G(Q)=\mathop{\oplus}\limits_{n\geq 0} Q^n/Q^{n+1}$ be the associated graded ring of $Q$. The fiber cone of $Q$ with respect to $K$ 
is the standard graded algebra $F_K(Q)=\mathop{\oplus}\limits_{n\geq 0} Q^n/KQ^n$. Let $\ell(M)$ denote the length of an $R$-module $M.$
The {\it{Hilbert function}} of the fiber cone $F_K(Q)$ is given by $H(F,n)=\ell(Q^n/KQ^n)$. It is well known that $H(F,n)$ agrees with
a polynomial $P(F,n)$ of degree $d-1$ for all $n\gg 0$, called the {\it{Hilbert polynomial}} of $F_K(Q)$. 
We can write  $P(F,n)$ in the following way:
\begin{equation*}
 P(F,n)=\sum\limits_{i=0}^{d-1}(-1)^if_i(Q)\binom{n+d-i-1}{d-1-i}
\end{equation*}
 where the coefficients $f_i(Q)$ are integers and are referred to as the {\it fiber coefficients} of $Q$ with respect to $K.$

The {\it{Hilbert-Samuel function}} of $Q$ is the function $H(Q,n)=\ell(R/Q^n)$.  
We recall the notion of Hilbert function of $Q$ with respect to $K$ from \cite{Jayanthan&Verma1}. 
It is the function $H_K(Q,-): \mathbb {Z} \rightarrow \mathbb{N}$ defined as
\begin{equation*} 
H_K(Q,n)=\ell(R/KQ^n) \text{~~~~~ for ~~~~} n\in \ints. 
\end{equation*}
It is known that for $n\gg 0$, $H(Q,n)$ (resp. $H_K(Q,n)$) agrees with a polynomial
$P(Q,n)$ (resp. $P_K(Q,n)$) of degree $d$. We can write these polynomials in the following manner:
\begin{eqnarray}
P(Q,n)&=& \sum\limits_{i=0}^d(-1)^ie_i(Q)\binom{n+d-i-1}{d-i}\label{Hilbert-S-poly}\\
P_K(Q,n)&=&\sum\limits_{i=0}^d(-1)^ig_i(Q)\binom{n+d-i-1}{d-i}\label{Hilbert-S-poly-for-K}
\end{eqnarray}
for unique integers $e_i(Q)$ (resp. $g_i(Q)$) known as the {\it{Hilbert coefficients}} of $Q$ (resp. {\it{Hilbert coefficients}} of $Q$ with 
respect to $K$).

In this paper, we relate the properties of the Hilbert coefficients $g_i(Q)$ with Cohen-Macaulayness of $R$. Jayanthan and Verma \cite{Jayanthan&Verma1}
have developed the basic tools for studying these coefficients which include the theory of $F_K(Q)$-superficial and regular elements, 
a version of Sally's machine for $F_K(Q)$ etc. 
Using these techniques, many
interesting properties of the polynomial $P_K(Q,n)$ and coefficients $g_i(Q)$ have been discussed in \cite{Jayanthan&Verma, Jayanthan&Verma1, Clare, Gu-Zhu-Tang, Gu-Zhu-Tang2}
as a generalization of analogous properties of $P(Q,n)$ and $e_i(Q)$.

Our work is inspired by the negativity conjecture posed by Vasconcelos
\cite[Conjecture 1]{vanishing-conjecture} at the conference in Yokohama in 
2008. He conjectured that for every parameter ideal $Q$ in a Noetherian local ring $R$, $e_1(Q)<0$ if and only if $R$ is not Cohen-Macaulay. Solving it partially, Mandal, Singh
and Verma \cite{Mandal-Singh-Verma} proved that if $\depth R=d-1$ then $e_1(Q)<0$ for every parameter ideal $Q$. They also showed that $e_1(Q)\leq 0$ in a Noetherian local ring. Ghezzi,
Hong and Vasconcelos in \cite{vanishing-partial1} proved the conjecture if $R$ is an integral
domain which is a homomorphic image of a Cohen-Macaulay ring. 
The conjecture is settled (more generally for modules) by Ghezzi et al. \cite{vanishing1, vanishing2} for unmixed local rings. Recall that $R$ 
is said to be {\it{unmixed}} if $\dim \hat{R}/p=\dim R$ for all $p\in Ass(\hat{R})$. The precise result is the following:

\begin{oldtheorem}\label{oldthm-thm3} {\rm\cite[Theorem 2.1]{vanishing1}}
 Let $(R,\m)$ be a Noetherian local ring of dimension $d>0$ and $Q$ a parameter ideal of $R$. Then the following statements are equivalent:
  \begin{enumerate}[label=(\alph*)]
   \item \label{thm-vanishing-statement-a} $R$ is Cohen-Macaulay;
   \item \label{thm-vanishing-statement-b} $R$ is unmixed and $e_1(Q)=0;$
   \item \label{thm-vanishing-statement-c} $R$ is unmixed and $e_1(Q)\geq 0$.
  \end{enumerate}
\end{oldtheorem}

By using the methods of \cite{vanishing1}, we obtain following characterization of Cohen-Macaulayness in terms of $g_0(Q)$ and $g_1(Q)$.
As a consequence of Theorem \ref{oldthm-thm1}, we recover Theorem \ref{oldthm-thm3}.

\begin{oldtheorem}\label{oldthm-thm1}{\rm[Theorem \ref{thm-vanishing}]}
  Let $(R,\m)$ be a Noetherian local ring of dimension $d>0$ and $Q$ a parameter ideal of $R$. Let $K$ be an ideal such that $Q\subseteq K$. Then the following statements are equivalent:
  \begin{enumerate}[label=(\alph*)]
   \item  $R$ is Cohen-Macaulay;
   \item  $R$ is unmixed and $g_0(Q)+g_1(Q)= -\ell(R/K)+\ell(R/Q);$
   \item  $R$ is unmixed and $g_0(Q)+g_1(Q)\geq -\ell(R/K)+\ell(R/Q);$
     \item $R$ is unmixed and $f_0(Q) \leq \ell(R/K)+e_1(Q)+e_0(Q)-\ell(R/Q);$
   \item $R$ is unmixed and $f_0(Q) = \ell(R/K)+e_1(Q)+e_0(Q)-\ell(R/Q).$
  \end{enumerate}
 \end{oldtheorem}

 This paper is organized as follows. Section \ref{section_pre} is devoted to some preliminary results on $F_K(Q)$-superficial elements.
In Section \ref{section_thevanishingthm}, 
we prove Theorem \ref{oldthm-thm1} and discuss its consequences. 
Throughout this paper $(R,\m)$ denotes a Noetherian local ring and $H_\m^i(*)$ denotes the $i$-th local cohomology functor with support in the maximal ideal $\m$.
 
We refer \cite{bruns-herzog} and \cite{Matsumura} for undefined terms.

\section{Preliminaries}\label{section_pre}
In this section, we recall and prove basic properties of superficial elements and 
superficial sequences in $G(Q)$ and $F_K(Q)$. The theory of superficial elements in $G(Q)$ is an 
effective method in the study of Hilbert coefficients, $e_i(Q)$, as it allows to apply
induction on the dimension of $R$. See \cite[Proposition 1.2]{book-Rossi-valla}. 
In \cite[Section 2]{Jayanthan&Verma1}, the authors developed the theory of superficial elements in $F_K(Q).$ 

Let $Q$ be an ideal of $R$ and $K$ an ideal with $Q\subseteq K.$ For an element $0\neq x\in R$,
let $x^*$(resp. $x^o$) denote the {\it initial form} of $x$ in $G(Q)$ (resp. $F_K(Q)$), i.e. 
the image of $x$ in $G(Q)_i$ (resp. $F_K(Q)_i$), where $i$ is the unique integer such that $x\in Q^i\setminus Q^{i+1}$ 
(resp. $x\in Q^i\setminus KQ^i$).
 \begin{definition}
  \begin{enumerate}
   \item For an element $x\in Q$ such that $x^*\neq 0$ in $G(Q)$, $x^*$ is said to be {\it $G(Q)$-superficial} 
   if there exists an integer $c >0$ such that $(Q^n:x)\cap Q^c=Q^{n-1}$ for all $n > c$. 
   \item For an element $x\in Q$ such that $x^o\neq 0$ in $F_K(Q)$,
   $x^o$ is said to be {\it $F_K(Q)$-superficial} if there exists an integer $c >0$
   such that $(0:x^o)\cap F_K(Q)_n=0$ for all $n>c$.
   \item For a sequence $x_1,\ldots,x_k \in Q$, $x_1^*,\ldots,x_k^* $ (resp. $x_1^o,\ldots,x_k^o$) is said to be {\it $G(Q)$ (resp. $F_K(Q)$)-superficial} if for all $i=1,\ldots,k$, $(x_i^\prime)^*$ 
   (resp. $(x_i^\prime)^o$) is $G(QR_{i-1})$ (resp. $F_{KR_{i-1}}(QR_{i-1})$)-superficial, where $R_{i-1}=R/(x_1,\ldots, x_{i-1})$ and 
   $x_i^\prime$ denotes the image of $x_i$ in $R_{i-1}$.
  \end{enumerate}
 \end{definition}

See \cite[Proposition 2.1]{Jayanthan&Verma} and \cite[Section 2]{Jayanthan&Verma1} for existence and basic
properties of $F_K(Q)$-superficial elements. We recall the following lemma from \cite{Jayanthan&Verma1} which provides a useful characterization of 
$F_K(Q)$-superficial elements. 

\begin{lemma}\label{lemma-sup-elt}{\rm\cite[Lemma 2.3]{Jayanthan&Verma1}}
Let $R$ be a Noetherian local ring of dimension $d >0$. Let $Q$ be an ideal of $R$ and $K$ an $\m$-primary ideal of $R$ such that
$Q\subseteq K$. Then the following statements hold.
\begin{enumerate}[label=(\alph*)] 
  \item \label{lemma-sup-elt-a} If there exists an integer $c>0$ such that $(KQ^{n}:x)\cap Q^c=KQ^{n-1}$ for all $n>c$, then $x^o$ is 
  $F_K(Q)$-superficial.
  \item \label{lemma-sup-elt-b} If $x^o$ is $F_K(Q)$-superficial and $x^*$ is $G(Q)$-superficial, then there exists an integer $c>0$ such that
  $(KQ^{n}:x)\cap Q^c = KQ^{n-1}$ for all $n > c$. Moreover if $x$ is regular on $R$, then $(KQ^n:x)=KQ^{n-1}$ for all $n\gg 0$.
 \end{enumerate}
\end{lemma}

Existence of $F_K(Q)$-superficial elements is guaranteed by \cite[Proposition 2.2]{Jayanthan&Verma1} when  $K$ is an $\m$-primary ideal. Indeed, we can choose 
$x\in Q$ such that $x^o$ is $F_K(Q)$-superficial as well as $x^*$ is $G(Q)$-superficial by \cite[Proposition 2.2]{Jayanthan&Verma1}.
Hence we use the characterization given by Lemma \ref{lemma-sup-elt} \ref{lemma-sup-elt-b} for $F_K(Q)$-superficial elements in this paper. 
The following lemma guarantees the existence of $F_K(Q)$-superficial elements avoiding a finite set of ideals not containing $Q.$ This result is well known in
case of $G(Q)$-superficial elements, see \cite[Corollary 8.5.9]{swanson-huneke}.

\begin{lemma}\label{lemma-existence-sup-elt-avoiding-ideals}
 Let $(R,\m)$ be a Noetherian local ring of dimension $d >0$. Let $Q$ be an ideal of $R$ and $K$ an  
 $\m$-primary ideal of $R$ such that $Q\subseteq K$. Let $I_1,\ldots,I_r$ be ideals in $R$ not containing $Q$. Then there exists
 an element $x\in Q\setminus \m Q$ that is not contained in any $I_i$ such that $x^*$ is $G(Q)$-superficial and $x^o$ is 
 $F_K(Q)$-superficial.
 
 In particular, if $Q$ contains a nonzerodivisor then there exists an element $x\in Q\setminus \m Q$ such that
 $x^*$ is $G(Q)$-superficial, $x^o$ is $F_K(Q)$-superficial and $x$ is a nonzerodivisor.
\end{lemma}
\begin{proof}
 The proof is similar to the proof of \cite[Corollary 8.5.9]{swanson-huneke}.
\end{proof}
The Hilbert coefficients $g_i(Q)$ behave nicely on reducing modulo $F_K(Q)$-superficial element. 
We need a refined version of \cite[Lemma 3.5]{Jayanthan&Verma1}. 

\begin{lemma}\label{lemma-gi-for-induction}
Let $(R,\m)$ be a Noetherian local ring of dimension $d >0$. Let $Q$ be an $\m$-primary ideal of $R$ and $K$ an ideal of $R$ such that $Q\subseteq K$.
Let $x\in Q$ such that $x^*$ is $G(Q)$-superficial and $x^o$ is $F_K(Q)$-superficial. Let $g_i(\bar{Q})$ denote the coefficients of the 
polynomial $P_{\bar{K}}(\bar{Q},n)$, where $\bar{Q}$ (resp. $\bar{K}$) denotes the image of an ideal $Q$ (resp. $K$) in $R/(x).$  Then
\begin{eqnarray*}
g_i(\bar{Q})= \begin{cases}
               g_i(Q) & \text{for}\ 0\leq i\leq d-2  \\
		g_{d-1}(Q)+(-1)^{d-1}\ell(0:x) & \text{for}\ i=d-1.
              \end{cases}
\end{eqnarray*}
\end{lemma}

\begin{proof}
 For $n\in \ints$, the exact sequence
 $$0\to \frac{KQ^n:x}{KQ^{n-1}}\to R/KQ^{n-1}\xrightarrow{x} R/KQ^{n}\to R/(KQ^n,x)\to 0$$
gives that $\ell(R/\bar{K}\bar{Q}^n)=\ell(R/KQ^n)-\ell(R/KQ^{n-1})+\ell((KQ^n:x)/KQ^{n-1})$ for all $n\in \ints$.
\\

 {\bf{Claim}} : $KQ^n\cap (0:x)=0$ and $(KQ^n:x)=KQ^{n-1}+(0:x)$ for $n\gg 0$.

 By Lemma \ref{lemma-sup-elt}\ref{lemma-sup-elt-b}, there exists an integer $c>0$ such that $(KQ^n:x)\cap Q^c=KQ^{n-1}$ for all $n>c$. So for all $n > c$,
 $KQ^n\cap (0:x)\subseteq \mathop{\cap}\limits_{m\geq n} (KQ^m :x)\cap Q^c = \mathop{\cap}\limits_{m\geq n} KQ^{m-1}=0.$
 By the Artin-Rees lemma, there exists an integer $k$ such that 
$KQ^n\cap (x)\subseteq (x)Q^{n-k}$ for all $n>k$.  
Let  $n >k+c$ and $y\in (KQ^n:x)$. Then $yx\in KQ^n\cap (x)\subseteq xQ^{n-k}$. Suppose $yx=zx$ for some $z\in Q^{n-k}$. 
 Hence $z\in (KQ^n:x)\cap Q^c=KQ^{n-1}$. Therefore $y=z+(y-z)\in KQ^{n-1}+(0:x)$.
  \\
  
Thus $\ell(R/\bar{K}\bar{Q}^n)=\ell(R/KQ^n)-\ell(R/KQ^{n-1})+\ell(0:x)$ for all $n\gg 0$. Hence, for $n \in \ints$,
\begin{equation} \label{equation:HilbpolyInRAndRbar}
P_{\bar{K}}(\bar{Q},n)= P_K(Q,n)-P_K(Q,n-1)+\ell(0:x). 
\end{equation}
Now, the result follows by comparing the coefficients of both sides of equation \eqref{equation:HilbpolyInRAndRbar}.
 \end{proof}
\begin{remark}
 If $x\in Q$ is a regular element, then we obtain \cite[Lemma 3.5]{Jayanthan&Verma1}. 
\end{remark}

We recall the following proposition from \cite{Clare} which determines the value of $g_i(Q)$ for parameter ideals in a Cohen-Macaulay local ring. 
There is a misprint in the statement of \cite[Theorem 7.2(2)]{Clare}. 
However the following version of the statement follows from the proof of \cite[Theorem 7.2]{Clare}. 
We provide an example below with $g_i(Q)=(-1)^i\ell(R/K)$ for $1\leq i\leq d$.

\begin{proposition}{\rm\cite[Theorem 7.2(2)]{Clare}}\label{prop-g1Q-in-CMring}
 Let $R$ be a Cohen-Macaulay local ring of dimension $d>0$ and $Q$ a parameter ideal of $R$. Let $K$ be an ideal such that $Q\subseteq K$. 
 Then $g_i(Q)=(-1)^i\ell(R/K)$ for $1\leq i\leq d$.
 \end{proposition}
 \begin{proof}
 Follows from comparing the coefficients of \cite[Equation 24]{Clare}.
\end{proof}
 \begin{example}\label{example-e1}
    Let $R=k[[x,y]]$ be a power series ring. Let $Q=(x^3,x^2y,y^3)$ and $K=\m^2$. Then $J=(x^3,y^3)$ is a minimal reduction of
  $Q$. We use Huneke's fundamental lemma \cite[Corollary 3.3]{Jayanthan&Verma1} to compute $g_0(Q)$ and $g_1(Q)$. Recall that 
  \begin{eqnarray*}
   g_1(Q)=&\sum\limits_{n\geq 1}v_n \mbox{~~and~~} g_2(Q)=\sum_{n\geq 1}(n-1)v_n+\ell(R/K) \mbox{~~where~~}\\
   v_n=&\begin{cases}
               e_0(Q) & \text{if}\ n= 0,\\
               e_0(Q)-\ell(R/KQ)+\ell(R/K) & \text{if}\ n=1,\\
               \ell(KQ^n/KJQ^{n-1})-\ell((KQ^{n-1}:J)/KQ^{n-2}) & \text{if}\ n \geq 2.
             \end{cases}
  \end{eqnarray*}
Therefore $$P_K(Q,n)=9\binom{n+1}{2}+3n+3.$$ So, we see that $\ell(R/K)=3=(-1)^ig_i(Q)$ for $i=1,2$.
 \end{example}
\begin{remark}\leavevmode
\noindent
1. Since $\ell(R/KQ^n)=\ell(R/Q^n)+\ell(Q^n/KQ^n)$ for all integers $n$, we have 
$P_K(Q,n)=P(Q,n)+P(F,n)$ for all integers $n.$ Thus comparing the coefficients of both sides, we get
\begin{eqnarray}
\label{equation relating g0 and f0} g_0(Q)&=&e_0(Q) \mbox{~~~and}\\ 
 f_i(Q)&=&e_{i+1}(Q)-g_{i+1}(Q)+e_i(Q)-g_i(Q) \text{~~for $0\leq i\leq d-1$}.\label{eqn-fi-gi-ei}
 \end{eqnarray}
 2. By putting $K=Q$ in \eqref{Hilbert-S-poly-for-K}, we see that $P_K(Q,n)=\ell(R/Q^{n+1})=P(Q,n+1)$ for all $n\gg 0$. Comparing the coefficients,
we get for $0\leq i\leq d$
\begin{equation}\label{eqn-g1-g0-e1}
 g_i(Q)=e_i(Q)-e_{i-1}(Q)+\ldots+(-1)^ie_0(Q).
\end{equation}
\end{remark} 
 
 \section{A characterization of Cohen-Macaulayness}\label{section_thevanishingthm}
 In this section, we first obtain a formula for the first Hilbert coefficient $g_1(Q)$ in a one dimensional Noetherian local ring.
 Then we give a characterization of Cohen-Macaulayness of an unmixed local ring in terms of the coefficients $g_0(Q)$ and $g_1(Q)$.
 This generalizes, in some sense, a result of Ghezzi et al. \cite[Theorem 2.1]{vanishing1}. We further discuss the results
 reminiscent of those in \cite[Section 2]{vanishing1}. 
 \begin{definition}
 Let $(0)=\mathop{\cap}\limits_{p\in\ass(R)}Q(p)$ be a primary decomposition of $(0)$ in $R$ and let
 $\assh(R)=\{p\in \ass(R)~|~\text{dim~} R/p=d\}$. The ideal  $U_R(0)=\mathop{\cap}\limits_{p\in\assh(R)}Q(p)$ is called the {\it unmixed component} of 
 $(0)$ in $R$. A ring $R$ is called {\it unmixed}
 if $\ass(\hat{R})=\assh(\hat{R})$ where $\hat{R}$ is the $\m$-adic completion of $R$. 
\end{definition} 

In \cite[Conjecture 1]{vanishing-conjecture}, Vasconcelos conjectured that $e_1(Q) < 0$ for every ideal $Q$ that is generated by 
system of parameters if and only if $R$ is not Cohen-Macaulay. 
In \cite[Theorem 2.1]{vanishing1} and \cite[Theorem 3.1]{vanishing2}, Ghezzi et al. proved that if $R$ is unmixed and $e_1(Q)\geq 0$ for some 
parameter ideal $Q$, then $R$ is Cohen-Macaulay. Motivated by this we ask:

\begin{question}\label{ques-1}
If $R$ is unmixed of dimension $d >0$ and $g_1(Q)\geq-\ell(R/K)$ for some parameter ideal $Q,$ then is $R$ Cohen-Macaulay ?
\end{question} 

The following example, which is worked out in \cite[Example 3.8]{Lori}, shows that the answer to Question \ref{ques-1} can be negative. 
\begin{example} \cite[Example 3.8]{Lori}\label{example-1}
 Let $R=k[[x^5,xy^4,x^4y,y^5]]\cong k[[t_1,t_2,t_3,t_4]]/J,$ where $J=(t_2t_3-t_1t_4,~t_2^4-t_3t_4^3,~t_1t_2^3-t_3^2t_4^2, ~t_1^2t_2^2-t_3^3t_4,
  ~ t_1^3t_2-t_3^4,~ t_3^5-t_1^4t_4)$. Then $\dim R=2$. Let $Q=(x^5,y^5).$ We have  
  $$P(Q,n)=5\binom{n+1}{2}+2n.$$
  Let $K=Q.$ Then, by \eqref{eqn-g1-g0-e1}, $g_1(Q)=e_1(Q)-e_0(Q)=-7 \geq  -\ell(R/K)$ but $\depth R=1.$ 
  \end{example}

\begin{remark}
The answer to Question \ref{ques-1} is affirmative for $K=\m.$ Indeed, suppose $K=\m$ and $g_1(Q)\geq -\ell(R/\m)=-1$ for some parameter ideal $Q.$ 
From \eqref{equation relating g0 and f0} and \eqref{eqn-fi-gi-ei}, we get that $e_1(Q)-g_1(Q)=f_0(Q)\geq 1$. Hence $e_1(Q)\geq 0$ 
which implies that $R$ is Cohen-Macaulay by \cite[Theorem 2.1]{vanishing1}. 
\end{remark}

For an arbitrary $K$, we prove in Theorem \ref{thm-vanishing} that 
an unmixed local ring is Cohen-Macaulay if $g_1(Q)\geq -\ell(R/K)+(\ell(R/Q)-g_0(Q))$ for some parameter ideal $Q\subseteq K$.
Since $\ell(R/Q)-g_0(Q)\geq 0$, Theorem \ref{thm-vanishing} is weaker than a complete solution of Question \ref{ques-1}.
We first discuss some results on $g_1(Q)$ for a parameter ideal $Q$ in a one dimensional ring. 

For an $\m$-primary ideal $I$ in $R$ and a finitely generated $R$-module $M$ of dimension $g,$ we write the {\it Hilbert-Samuel polynomial} of $M$ 
with respect to $I$ as
\begin{eqnarray*}
  P_M(I,n)=\sum\limits_{i=0}^{g}(-1)^i e_i(I,M)\binom{n+g-i-1}{g-i}.
\end{eqnarray*}

\begin{proposition}\label{prop-g1-in-d=1}
   Let $(R,\m)$ be a Noetherian local ring of dimension one and $Q$ a parameter ideal of $R$. Let $K$ be an $\m$-primary ideal.
   Then $g_1(Q)=-\ell(R/K)-\ell(H_\m^0(K))$.
 \end{proposition}
\begin{proof} 
 The short exact sequence 
 $$0 \to K \to R \to R/K \to 0$$
induces a surjective map from $H_{\m}^1(K)$ to $H_{\m}^1(R).$ Since $H_{\m}^1(R) \neq 0,$ we get that $H_{\m}^1(K) \neq 0.$
Thus $K$ is an $R$-module of dimension one. 
For all integers $n,$ we have 
 $$\ell(R/KQ^n)=\ell(R/K)+\ell(K/KQ^n).$$ 
Therefore 
 \begin{equation*}\label{eqn-e3}
  g_1(Q)=-\ell(R/K)+e_1(Q,K)=-\ell(R/K)-\ell(H_\m^0(K))
 \end{equation*}
where the last equality follows from 
 \cite[Proposition 3.1]{Mandal-Singh-Verma}.
\end{proof}

\begin{proposition}\label{prop-lemma for main result}
 Let $(R,\m)$ be a Noetherian local ring of dimension one and $Q$ a parameter ideal of $R$. Let $K$ be an ideal
  such that $Q\subseteq K$. Suppose $g_0(Q)+g_1(Q)\geq -\ell(R/K)+\ell(R/Q)$. Then $R$ is Cohen-Macaulay.
\end{proposition}
\begin{proof}
 Set $W:=H_\m^0(R)$ and $\bar{R}:=R/W$. Then $\bar{R}$ is a Cohen-Macaulay local ring of dimension one. Consider the following exact sequence
 $$0\to W/(KQ^n\cap W)\to R/KQ^n\to \bar{R}/KQ^n\bar{R}\to 0.$$
 By the Artin-Rees lemma, there exists an integer $k$ such that for all $n\gg 0$, $KQ^n\cap W\subseteq Q^n\cap W\subseteq Q^{n-k}W=0$. Hence for all $n\gg 0$, 
 \begin{eqnarray}
  \ell(R/KQ^n)&=&\ell(\bar{R}/KQ^n\bar{R})+\ell(W).  \nonumber
  \end{eqnarray}
  Thus 
  \begin{eqnarray}\label{eqn-R-to-Rbar}
   P_K(Q,n)=P_{\bar{K}}(\bar{Q},n)+\ell(W) \mbox{ for all integers }n.
  \end{eqnarray}
This implies that 
  \begin{eqnarray*}
  g_0(Q)&=&g_0(Q\bar{R}) \text{~~and~~} g_1(Q)=g_1(Q\bar{R})-\ell(W)\label{g1inbarR}\label{g1-in-d=1}.
  \end{eqnarray*}
  Therefore
  \begin{eqnarray} 
  g_0(Q\bar{R})+g_1(Q\bar{R})&\geq& -\ell(R/K)+\ell(R/Q)+\ell(W) \label{lb on g_0+g_1}.
  \end{eqnarray}
Since $\bar{R}$ is Cohen-Macaulay, using \eqref{equation relating g0 and f0}, we get that $g_0(Q\bar{R})=e_0(Q\bar{R})=\ell(\bar{R}/Q\bar{R}).$ 
By Proposition \ref{prop-g1Q-in-CMring}, $g_1(Q\bar{R})=-\ell(\bar{R}/K\bar{R})$.
Hence
\begin{eqnarray}
g_0(Q\bar{R})+g_1(Q\bar{R})&=& \ell(\bar{R}/Q\bar{R})-\ell(\bar{R}/K\bar{R}))\nonumber\\
 &=&\ell(R/(W+Q))-\ell(R/(W+K))=\ell(K/(Q+W)\cap K))\nonumber\\
 &\leq & \ell(K/Q)\nonumber\\
 &=& -\ell(R/K)+\ell(R/Q). \label{ub on g_0+g_1}
\end{eqnarray}
Comparing equations \eqref{lb on g_0+g_1} and \eqref{ub on g_0+g_1}, we get that $\ell(W)\leq 0$ which implies that $W=0.$ Thus $R$ is Cohen-Macaulay. 
\end{proof}

In order to prove Theorem \ref{thm-vanishing}, we need an analogue of \cite[Lemma 2.3]{vanishing1}.  

\begin{lemma}\label{lemma-for-g1-in-unmixed-rings}
 Let $(R,\m)$ be a Noetherian local ring of dimension $ d>0$ and $Q$ a parameter ideal of $R$. Let $K$ be an ideal
  such that $Q\subseteq K$. Suppose $U=U_R(0)\neq 0$ and $S=R/U$. Then the following assertions hold:
   \begin{enumerate}[label=(\alph*)]
   \item \label{lemma-for-g1-in-unmixed-rings-a} $\dim U<\dim R$.
   \item \label{lemma-for-g1-in-unmixed-rings-b} we have $g_0(Q)=g_0(QS)$ and  
   \begin{equation*}
    g_1(Q)=
    \begin{cases}
      g_1(QS) & \text{if }\dim U\leq d-2 \\
      g_1(QS)-s_0 & \text{if }\dim U=d-1,
    \end{cases}
  \end{equation*}
  where $s_0$ is the multiplicity of the module $\mathop{\oplus}\limits_{n\geq 0}U/(KQ^{n+1}\cap U)$.
   \item \label{lemma-for-g1-in-unmixed-rings-c} $g_1(Q)\leq g_1(QS)$ with equality if and only if $\dim U\leq d-2$.
   \end{enumerate}
\end{lemma}
\begin{proof}
\ref{lemma-for-g1-in-unmixed-rings-a} 
Let $(0)=\mathop{\cap}\limits_{p\in\ass(R)}Q(p)$ be a primary decomposition of $(0)$ and 
$U=U_R(0)=\mathop{\cap}\limits_{p\in\assh(R)}Q(p).$ For all $p\in \ass(R)$ and $p^\prime\in \assh(R)$ with $p\neq p'$, we have $Q(p)R_{p'}=R_{p'}.$ Hence for all $p'\in \assh(R)$, we get that $(0)=\mathop{\cap}\limits_{p\in\ass(R)}Q(p)R_{p'}=Q(p')R_{p'}.$ Thus $UR_{p'}=\mathop{\cap}\limits_{p\in\assh(R)}Q(p)R_{p'}=Q(p')R_{p'}=(0).$
Hence $\assh(R)\cap \Supp_R(U)=\phi.$ Therefore $\dim U < \dim R$.

\ref{lemma-for-g1-in-unmixed-rings-b}
Considering the short exact sequence 
$$0\to U/(KQ^n\cap U)\to R/KQ^n\to S/KQ^nS\to 0,$$
we get that
\begin{equation}\label{eqn-lemma-1}
\ell(R/KQ^n)=\ell(S/KQ^nS)+\ell(U/(KQ^n\cap U))    \mbox{ for all } n\in \ints.      
\end{equation}
Hence $\ell(U/(KQ^n\cap U))$ agrees  with a polynomial, say $T(n)$ for $n\gg 0$. We write 
 \begin{equation}\label{eqn-lemma-2}
 T(n)=s_0\binom{n+t}{t}-s_1\binom{n+t-1}{t-1}+\ldots+(-1)^ts_t
 \end{equation}
for some $t\geq 0$ and $s_i \in \mathbb Z$ for $0\leq i\leq t.$ We claim that $t=\dim U.$ 
By the Artin-Rees lemma, there exists an integer $k$ such that for all $n\gg0$,
$KQ^n\cap U\subseteq Q^n\cap U= Q^{n-k}(Q^{k}\cap U)\subseteq Q^{n-k}U.$ Hence $\ell(U/Q^{n-k}U) \leq \ell(U/(KQ^n\cap U))$ for all $n\gg 0$ which implies
$t\geq \dim U$.
On the other hand, Since $\ell(U/KQ^nU)=\ell(U/Q^nU)+\ell(Q^nU/KQ^nU)$, 
we see that $\ell(U/KQ^nU)$ coincides with a polynomial of degree equals $\dim U$ for all $n\gg 0$.  Therefore $\ell(U/(KQ^n\cap U))\leq \ell(U/KQ^nU)$ for all 
$n\in \ints$ implies
that $t=\dim U$. 

From \eqref{eqn-lemma-1}, we get
\begin{equation}\label{eqn-lemma-3}
P_K(Q,n)=P_{KS}(QS,n)+T(n) \mbox{ for all } n \in \mathbb{Z}.
\end{equation}
By comparing the coefficients of both sides of \eqref{eqn-lemma-3} and using \eqref{eqn-lemma-2}, we get the result.

\ref{lemma-for-g1-in-unmixed-rings-c} Follows from \ref{lemma-for-g1-in-unmixed-rings-b}.
\end{proof}

The proof of Theorem \ref{thm-vanishing} is based on the methods employed in \cite{vanishing1}. 
We recall the following results from \cite{tightclosure} which are needed to prove Theorem \ref{thm-vanishing}.  

\begin{lemma}\label{thm-as-the-main-lemma1}{\rm\cite[Lemma 3.1]{tightclosure}}
 Let $(R,\m)$ be a complete local ring. Suppose $\ass(R)\subseteq\assh(R)\cup \{\m\}$. Then $H_\m^1(R)$ has  finite length.
\end{lemma}

Using Lemma \ref{lemma-existence-sup-elt-avoiding-ideals}, the proof of \cite[Proposition 3.3]{tightclosure} shows that $x_1,\ldots,x_d$ in \cite[Proposition 3.3]{tightclosure} can be choosen such that $x_1^o,\ldots,x_d^o$ is $F_K(Q)$-superficial and $x_1^*,\ldots,x_d^*$ is $G(Q)$-superficial. 

\begin{proposition}\label{thm-as-the-main-lemma2}{\rm\cite[Proposition 3.3]{tightclosure}}
 Let $(R,\m)$ be a homomorphic image of a Cohen-Macaulay local ring of dimension $d$ and assume
 that $\ass(R)\subseteq \assh(R)\cup \{\m\}$. Let 
 $Q$ be a parameter ideal. Then there exists a system of generators $x_1,\ldots,x_d$ of
 $Q$ such that $x_1^o,\ldots,x_d^o$ is $F_K(Q)$-superficial, $x_1^*,\ldots,x_d^*$ is $G(Q)$-superficial and   $\ass(R/Q_i)\subseteq \assh(R/Q_i)\cup \{\m\},$
where $Q_i=(x_1,\ldots,x_i)$ for $0\leq i\leq d.$
\end{proposition}

We now prove the main theorem of this section.

\begin{theorem}\label{thm-vanishing}
  Let $(R,\m)$ be a Noetherian local ring of dimension $d >0$ and $Q$ a parameter ideal of $R$. Let $K$ be an ideal such that $Q\subseteq K$. 
  Then the following statements are equivalent:
  \begin{enumerate}[label=(\alph*)]
   \item \label{thm-vanishing-statement-a} $R$ is Cohen-Macaulay;
   \item \label{thm-vanishing-statement-b} $R$ is unmixed and $g_0(Q)+g_1(Q)= -\ell(R/K)+\ell(R/Q);$
   \item \label{thm-vanishing-statement-c} $R$ is unmixed and $g_0(Q)+g_1(Q)\geq -\ell(R/K)+\ell(R/Q);$
   \item \label{thm-vanishing-statement-d} $R$ is unmixed and $f_0(Q) \leq \ell(R/K)+e_1(Q)+e_0(Q)-\ell(R/Q);$
   \item \label{thm-vanishing-statement-e} $R$ is unmixed and $f_0(Q) = \ell(R/K)+e_1(Q)+e_0(Q)-\ell(R/Q).$
  \end{enumerate}
 \end{theorem}

\begin{proof}
Using \eqref{equation relating g0 and f0} and \eqref{eqn-fi-gi-ei}, we get that \ref{thm-vanishing-statement-e} $\Leftrightarrow$ \ref{thm-vanishing-statement-b} and
\ref{thm-vanishing-statement-d} $\Leftrightarrow$ \ref{thm-vanishing-statement-c}. 
Hence it suffices to prove \ref{thm-vanishing-statement-a} $\Rightarrow$ \ref{thm-vanishing-statement-b} 
$\Rightarrow$ \ref{thm-vanishing-statement-c} $\Rightarrow$ \ref{thm-vanishing-statement-a}. \\
 \ref{thm-vanishing-statement-a} $\Rightarrow$ \ref{thm-vanishing-statement-b} Follows from Proposition \ref{prop-g1Q-in-CMring} and the fact that $g_0(Q)=e_0(Q)=\ell(R/Q)$.\\
 \ref{thm-vanishing-statement-b} $\Rightarrow$ \ref{thm-vanishing-statement-c} Clear. \\
 \ref{thm-vanishing-statement-c} $\Rightarrow$ \ref{thm-vanishing-statement-a} We prove by induction on $d.$ The result is clear for $d=1.$ Let $d \geq 2.$
 
 We may assume that $R$ is complete with infinite residue field.
 Suppose $d=2$. Then we may assume that $Q=(x_1,x_2)$ such that 
 $x_1^o$ is $F_K(Q)$-superficial. Since $R$ is unmixed, we can choose 
 $x_1$ to be a nonzerodivisor on $R$. Let $S=R/(x_1).$ Then $Q/(x_1)$ is a parameter ideal of $S.$ By Lemma \ref{lemma-gi-for-induction}, $g_0(QS)=g_0(Q)$ and 
 $g_1(QS)=g_1(Q).$ Hence 
 \begin{equation*}
  g_0(QS)+g_1(QS)\geq -\ell(R/K)+\ell(R/Q)= -\ell(S/KS)+\ell(S/QS).
 \end{equation*}
Therefore, by Proposition \ref{prop-lemma for main result}, $S$ is Cohen-Macaulay which implies that $R$ is Cohen-Macaulay.

Let $d\geq 3.$ By Theorem \ref{thm-as-the-main-lemma2}, there exists a system of generators $x_1,\ldots,x_d$ of $Q$ such that $x_1^o$ is $F_K(Q)$-superficial and 
$\ass(R/(x_1))\subseteq\assh(R/(x_1))\cup \{\m\}$. Let $S=R/(x_1)$ and $\bar{S}=S/U_S(0)$. Then $\bar{S}$ is an unmixed local ring of
dimension $d-1$ and $Q\bar{S}$ is a parameter ideal contained in $K\bar{S}$. Since $\dim U_S(0)=0$, by Lemma \ref{lemma-for-g1-in-unmixed-rings}\ref{lemma-for-g1-in-unmixed-rings-b}, we get that $g_i(Q\bar{S})=g_i(QS)$ for $i=0,1.$ Therefore
\begin{eqnarray*}
g_0(Q\bar{S})+g_1(Q\bar{S})&=&g_0(Q)+g_1(Q) \\
&\geq&  -\ell(R/K)+\ell(R/Q)\\
&=&-\ell(S/KS)+\ell(S/QS)\\
&\geq& \ell(KS/((QS+U_S(0))\cap KS))=\ell((KS+U_S(0))/(QS+U_S(0)))\\
&=& -\ell(S/(KS+U_S(0)))+\ell(S/(QS+U_S(0)))\\
&=&-\ell(\bar{S}/K\bar{S})+\ell(\bar{S}/Q\bar{S}).
\end{eqnarray*}
Hence by induction hypothesis, $\bar{S}$ is Cohen-Macaulay. Therefore $H_\m^i(\bar{S})=0$ for $0\leq i\leq d-2$. The exact sequence
 $$ 0\to U_S(0)\to S\to \bar{S}\to 0$$ 
 gives the following long exact sequence 
 $$ \cdots \to H_\m^i(U_S(0))\to H_\m^i(S)\to H_\m^i(\bar{S})\to \cdots\ .$$
 Since $U_S(0)$ is Artinian, $H_\m^0(U_S(0))=U_S(0)$ and $H_\m^i(U_S(0))=0$ for all $i\geq 1.$ Therefore 
 $H_\m^0(S)=U_S(0)$ and $H_\m^i(S)=0$ for $1\leq i\leq d-2.$

Now considering the exact sequence 
$$0\to R\xrightarrow{x_1} R \to S \to 0,$$ 
we get the long exact sequence 
$$\cdots \to H_\m^{i-1}(S)\to  H_\m^{i}(R)\xrightarrow{x_1} H_\m^{i}(R) \to H_\m^{i}(S)  \to \cdots .$$

This implies that the map $H_\m^1(R)\xrightarrow{x_1} H_\m^1(R)$ is surjective and $ H_\m^{i}(R)\xrightarrow{x_1} H_\m^{i}(R)$ is injective, for $2\leq i\leq d-1.$ 
Thus $H_\m^1(R)=x_1 H_\m^1(R).$ Since $H_\m^1(R)$ is finitely generated by Theorem \ref{thm-as-the-main-lemma1}, using Nakayama's Lemma, we get that $H_\m^1(R)=0.$ 
Since $H_\m^{i}(R)$ is $\m$-torsion, the injectivity of the map $ H_\m^{i}(R)\xrightarrow{x_1} H_\m^{i}(R)$ gives that $H_\m^{i}(R)=0$ for $2\leq i\leq d-1$. Therefore $R$ is Cohen-Macaulay.
\end{proof} 

As a consequence we recover the result of Ghezzi et al. \cite[Theorem 2.1]{vanishing1}.

\begin{corollary}
 Let $R$ be a Noetherian local ring of dimension $d>0$ and $Q$ a parameter ideal of $R$. Then the following statements are equivalent:
 \begin{enumerate}[label=(\alph*)]
  \item \label{corr-vanishing-statement-a} $R$ is Cohen-Macaulay;
  \item \label{corr-vanishing-statement-b} R is unmixed and $e_1(Q)=0$;
  \item \label{corr-vanishing-statement-c} $R$ is unmixed and $e_1(Q)\geq 0$.
 \end{enumerate}
\end{corollary}
\begin{proof}

Let $K=Q$. From \eqref{eqn-g1-g0-e1}, $g_0(Q)+g_1(Q)= e_1(Q)$. Hence the result follows from Theorem \ref{thm-vanishing}.
\end{proof}

In case of $d=1$, we have $g_1(Q)\leq -\ell(R/K)$ by Proposition \ref{prop-g1-in-d=1}. We obtain the following corollary
in this direction for $d>0$. In particular for $K=Q$, we recover the non-positivity of the Chern number $e_1(Q)$
\cite[Corollary 2.4(a)]{vanishing1} in Corollary \ref{corr-recover-e1}.

\begin{corollary}\label{corr-negativity-result-for-g1}
 Let $(R,\m)$ be a Noetherian local ring of dimension $d > 0$ and $Q$ a parameter ideal of $R$. Let $K$ be an ideal such that $Q\subseteq K$. Then
 \begin{enumerate} 
            \item  \label{corr-negativity-result-for-g1-aa} 
                   \begin{enumerate} [label=(\alph*)]
                        \item \label{corr-negativity-result-for-g1-a} $g_0(Q)+ g_1(Q)\leq -\ell(R/K)+\ell(R/Q)$.
                         \item  $f_0(Q) \geq \ell(R/K)+ e_1(Q)+e_0(Q)-\ell(R/Q)$.
                    \end{enumerate}
            \item  \label{corr-negativity-result-for-g1-bb} Suppose $\depth R=d-1$. Then
                   \begin{enumerate}[label=(\alph*)]
			  \item  \label{corr-negativity-result-for-g1-b}  $g_0(Q) + g_1(Q)< -\ell(R/K)+\ell(R/Q)$.
                          \item $f_0(Q) > \ell(R/K)+ e_1(Q)+e_0(Q)-\ell(R/Q)$.
                         
                     \end{enumerate}
 \end{enumerate}
\end{corollary}
\begin{proof}
In view of \eqref{equation relating g0 and f0} and \eqref{eqn-fi-gi-ei}, it suffices to prove \ref{corr-negativity-result-for-g1-aa}\ref{corr-negativity-result-for-g1-a} and
\ref{corr-negativity-result-for-g1-bb}\ref{corr-negativity-result-for-g1-b}. \\
\ref{corr-negativity-result-for-g1-aa}\ref{corr-negativity-result-for-g1-a} We may assume that $R$ is complete.
 Let $d=1.$ Since $g_0(Q)=e_0(Q) \leq \ell(R/Q),$ using Proposition \ref{prop-g1-in-d=1}, we get that  $g_0(Q)+ g_1(Q)\leq -\ell(R/K)+\ell(R/Q).$ 
 Suppose $d\geq 2.$ Set $S=R/U_R(0).$ Then 
 $S$ is an unmixed local ring and $QS$ is a parameter ideal of $S.$ 
 Hence
 \begin{eqnarray*}
  g_0(Q)+ g_1(Q)&\leq& g_0(QS)+ g_1(QS) \hspace*{1in}\mbox{ [by Lemma \ref{lemma-for-g1-in-unmixed-rings}\ref{lemma-for-g1-in-unmixed-rings-b}] }\\
  &\leq& -\ell(S/KS)+\ell(S/QS) \hspace{0.7in}\mbox{ [from Theorem \ref{thm-vanishing}] }\\
 & \leq& -\ell(R/K)+\ell(R/Q).
 \end{eqnarray*}
 
\ref{corr-negativity-result-for-g1-bb}\ref{corr-negativity-result-for-g1-b}
Let $d=1$. Then $\depth(R)=0$ implies that $g_0(Q)=e_0(Q)<\ell(R/Q)$. Hence, by Proposition \ref{prop-g1-in-d=1}, $g_0(Q) + g_1(Q)< -\ell(R/K)+\ell(R/Q)$.
Suppose $d\geq 2$. Let $Q=(x_1,\ldots,x_d)$ such that $x_1^o,\ldots,x_{d-1}^o$ is an $F_K(Q)$-superficial sequence. Since $\depth(R)=d-1$, 
we may choose $x_1,\ldots,x_{d-1}$ to be an $R$-regular sequence. Let `` $\bar{}$ '' denote reduction modulo $x_1,\ldots,x_{d-1}$. Then, 
by Lemma \ref{lemma-gi-for-induction},
$g_0(Q) + g_1(Q)=g_0(Q\bar{R}) + g_1(Q\bar{R}) < -\ell(R/K)+\ell(R/Q)$ by induction hypothesis.
 \end{proof}

 \begin{corollary}\label{corr-recover-e1}
 Let $(R,\m)$ be a Noetherian local ring of dimension $d > 0$ and $Q$ a parameter ideal of $R$. Then
 \begin{enumerate}[label=(\alph*)]
                          \item $e_1(Q)\leq 0$
                           \item If $\depth R=d-1$, then $e_1(Q) <0$.
 \end{enumerate}
\end{corollary}
\begin{proof}
 It follows from letting $K=Q$ in Corollary \ref{corr-negativity-result-for-g1} and using \eqref{eqn-g1-g0-e1}.
\end{proof}
In the following example, we have the inequality as stated in part \ref{corr-negativity-result-for-g1-aa}\ref{corr-negativity-result-for-g1-a} of 
Corollary \ref{corr-negativity-result-for-g1}. 
\begin{example}\label{example-e3}
 Let $S=k[[X,Y,Z]]]$ be a power series ring over a field $k$ and $I=(YZ,X^2Y,Y^3)$. Consider the ring $R=S/I=k[[x,y,z]]$. Then $\dim R=2$ and $\depth R=0$. Let $Q=(x,z^2)$ and $K=(x,y,z^2)$. Notice that $H_{\m}^0(R)=(Y)/I$, hence $\bar{R}=R/H_{\m}^0(R)=k[[X,Z]]$ which is Cohen-Macaulay.
By \eqref{eqn-R-to-Rbar} and Proposition \ref{prop-g1Q-in-CMring}, we get $g_0(Q)+g_1(Q)=g_0(\bar{Q})+g_1(\bar{Q})=\ell(\bar{R}/Q\bar{R})-\ell(\bar{R}/K\bar{R})=2-2=0$ whereas
$-\ell(R/K)+\ell(R/Q)=-2+4=2$. 
\end{example}

The following two results can be seen as analogs of \cite[Theorem 2.6, Corollary 2.7]{vanishing1}. Corollary \ref{corr-cm-gi} gives a characterization of Cohen-Macaulayness of $R$ in terms of $g_i(Q)$.

\begin{theorem}\label{thm-vanishing-main}
 Let $(R,\m)$ be a Noetherian local ring of dimension $d > 0$ and $Q$ a parameter ideal of $R.$ Let $K$ be an ideal such that $Q\subseteq K.$ 
 Suppose $R$ is a homomorphic image of a Cohen-Macaulay ring. Let $U=U_R(0).$ Then the following are equivalent:
  \begin{enumerate}[label=(\alph*)]
   \item \label{thm-vanishing-main-a} $g_0(Q)+ g_1(Q)= -\ell(R/(K+U))+\ell(R/(Q+U));$
   \item \label{thm-vanishing-main-b} $R/U$ is Cohen-Macaulay and $\dim U\leq d-2$.
  \end{enumerate}
\end{theorem}
\begin{proof}
 \ref{thm-vanishing-main-a} $\Rightarrow$ \ref{thm-vanishing-main-b} Let $S=R/U.$ Then $S$ is an unmixed local ring. If $U=0$, then $R$ is unmixed.
 Hence $g_0(Q)+ g_1(Q)= -\ell(R/K)+\ell(R/Q)$  
 implies that $R$ is Cohen-Macaulay, by Theorem \ref{thm-vanishing}. Suppose $U\neq 0$. 
 First we show that $\dim U\leq d-2.$ By Lemma  \ref{lemma-for-g1-in-unmixed-rings} \ref{lemma-for-g1-in-unmixed-rings-a}, $\dim U \leq d-1.$ Suppose $\dim U=d-1.$ 
 Then $g_0(Q)=g_0(QS)$ and $g_1(Q)=g_1(QS)-s_0$ for some $s_0 \geq 1,$ by  Lemma  \ref{lemma-for-g1-in-unmixed-rings} \ref{lemma-for-g1-in-unmixed-rings-b}. 
 Therefore
 \begin{eqnarray*}
  g_0(Q)+ g_1(Q)&<& g_0(QS)+g_1(QS)\\
  &\leq& -\ell(S/KS)+\ell(S/QS) \hspace*{0.7in} \mbox{ [by Corollary \ref{corr-negativity-result-for-g1}] }\\
&=& -\ell(R/(K+U))+\ell(R/(Q+U))
\end{eqnarray*}
which is a contradiction. Hence $\dim U\leq d-2.$ By Lemma \ref{lemma-for-g1-in-unmixed-rings} \ref{lemma-for-g1-in-unmixed-rings-b}, $g_1(Q)=g_1(QS).$ 
Hence $g_0(QS)+g_1(QS) = -\ell(S/KS)+\ell(S/QS).$
Therefore, by Theorem \ref{thm-vanishing}, $S=R/U$ is Cohen-Macaulay.

\ref{thm-vanishing-main-b} $\Rightarrow$ \ref{thm-vanishing-main-a} Let $S=R/U.$ Since $\dim U\leq d-2$, by  Lemma \ref{lemma-for-g1-in-unmixed-rings}\ref{lemma-for-g1-in-unmixed-rings-b}, 
$g_1(Q)=g_1(QS)=-\ell(S/KS),$ where the last equality holds by Proposition \ref{prop-g1Q-in-CMring}. Therefore 
\begin{eqnarray*}
 g_0(Q)+ g_1(Q)&=&\ell(S/QS)-\ell(S/KS)\\
 &=&\ell(R/(Q+U))-\ell(R/(K+U)).
\end{eqnarray*}
\end{proof}

\begin{corollary}\label{corr-cm-gi}
 Let $(R,\m)$ be a Noetherian local ring of dimension $ d >0$.  Suppose $R$ is a homomorphic image of a Cohen-Macaulay ring.
 Let $Q$ be a parameter ideal of $R$ and $K$ an ideal such that $Q\subseteq K$. Let $U=U_R(0)$. Suppose 
  \begin{equation*}
   g_i(Q)=(-1)^i(g_0(Q)-\ell(R/(Q+U))+\ell(R/(K+U)))
  \end{equation*}
for $1\leq i\leq d$. Then $R$ is Cohen-Macaulay.
\end{corollary}
\begin{proof}
Since 
$g_1(Q)=-(g_0(Q)-\ell(R/(Q+U))+\ell(R/(K+U))),$ by Theorem \ref{thm-vanishing-main}, $R/U$ is Cohen-Macaulay and $\dim U\leq d-2$.
Set $S=R/U$. By Lemma \ref{lemma-for-g1-in-unmixed-rings} \ref{lemma-for-g1-in-unmixed-rings-b}, $g_0(Q)=g_0(QS)=\ell(S/QS)=\ell(R/(Q+U)).$ 
From Proposition \ref{prop-g1Q-in-CMring}, we have  
\begin{eqnarray*} 
g_i(QS)=(-1)^i\ell(S/KS)\mbox{ for } 1\leq i\leq d.
\end{eqnarray*}
Hence 
$$g_i(Q)=(-1)^i\ell(R/(K+U))=(-1)^i\ell(S/KS) \mbox{ for } 1\leq i\leq d.$$
Therefore, for $n\gg 0,$ 
 \begin{eqnarray*}
     \ell(S/KQ^nS)&=&\ell(S/QS)\binom{n+d-1}{d}+\ell(S/KS)\binom{n+d-2}{d-1}+\cdots+\ell(S/KS) \mbox{ and }\\
   \ell(R/KQ^n)&=&\ell(S/QS)\binom{n+d-1}{d}+\ell(S/KS)\binom{n+d-2}{d-1}+\cdots+\ell(S/KS).
    \end{eqnarray*}
Thus 
\begin{equation*}
           \ell(U/(KQ^n\cap U))=\ell(R/KQ^n)-\ell(S/KQ^nS)=0 \mbox{~~for~~} n\gg0 .
          \end{equation*}
 This implies that $U=0$ and hence $R$ is Cohen-Macaulay.         
 \end{proof}

\section*{Acknowledgements}
The author is grateful to her advisor Anupam Saikia for the encouragement to pursue this work. She is also
grateful to Krishna Hanumanthu and Shreedevi K. Masuti for insightful discussions. She thanks the Indian Institute of Technology, Guwahati
for granting the Ph.D. scholarship and the Institute of Mathematical Sciences, Chennai for its hospitality where a 
significant part of this work was done.

\end{document}